\documentclass[12pt]{amsart}
\usepackage{amssymb}
\usepackage{amsfonts,amscd}
\usepackage{amsmath}
\usepackage[backref=page]{hyperref}
 \numberwithin {equation}{section}
\newtheorem{theo}{Theorem}[section]

\newtheorem{corol}[theo]{Corollary}
\newtheorem{prop}[theo]{Proposition}
\theoremstyle{definition}
\newtheorem{remark}[theo]{Remark}

\newtheorem{example}[theo]{Example}

\newtheorem{defi}[theo]{Definition}
\newcommand{\dom}{\operatorname{dom}}
\newcommand{\pq}{\preccurlyeq}
\newcommand{\sms}{\smallsetminus}

\begin{document}
\title{Ekeland variational principle and its equivalents in  $T_1$-quasi-uniform spaces}
\author{S. Cobza\c{s}}
\address{Babe\c{s}-Bolyai University, Department of Mathematics, Cluj-Napoca, Romania}
\email{scobzas@math.ubbcluj.ro}

\begin{abstract}
  The present paper is concerned with Ekeland Variational Principle (EkVP) and its equivalents (Caristi-Kirk fixed point theorem, Takahashi minimization principle,
Oettli-Th\'era equilibrium version of EkVP) in quasi-uniform spaces. These extend some results proved by  Hamel and L\"ohne (Kluwer 2003) and Hamel, Nonlinear Anal. \textbf{62} (2005), 913--924,
  in uniform spaces, as well as those proved in quasi-metric spaces by various authors. The case of  $F$-quasi-gauge spaces, a non-symmetric version of $F$-gauge spaces introduced by   Fang, J. Math. Anal. Appl. \textbf{202} (1996),   398--412, is also considered. The paper ends with the quasi-uniform versions of some minimization principles proved by Arutyunov and Gel'man, Zh. Vychisl. Mat. Mat. Fiz. \textbf{49} (2009),  1167--1174, and Arutyunov, Proc. Steklov Inst. Math. \textbf{291} (2015), no.~1, 24--37, in complete metric spaces.

  \textbf{MSC 2020}:\; 58E30, 54E15, 47H10, 90C33

  \textbf{Key words:}  Ekeland Variational Principle, Takahashi minimization principle, equilibrium problems, uniform spaces, quasi-uniform spaces, gauge spaces,
 quasi-gauge spaces, completeness in quasi-uniform spaces,
Ekeland Variational Principle, Takahashi minimization principle, equilibrium problems, quasi-uniform spaces,  quasi-gauge spaces, completeness in quasi-uniform spaces

  \end{abstract}

\date{\today}
\maketitle

\section{Introduction}

Ekeland Variational Principle (EkVP), announced by Ekeland in \cite{ek72}, with the proof published in \cite{ek74}, asserts that a small perturbation of a lower semicontinuous (lsc) function $f$ defined on a complete metric space $X$ attains its strict minimum. It turned out that the completeness of $X$ is also necessary for the  validity of EkVP
(see \cite{cobz18}  for various results forcing completeness). Also it is equivalent to many results in fixed point theory (Caristi-Kirk fixed point theorem), geometry of Banach spaces (Dane\v{s} drop theorem, Bishop-Phelps  result on the density of support functionals) and has many applications to  optimization, differentiation of convex and Lipschitz functions, economics, biology, etc. Some of these applications are presented in Ekeland's paper \cite{ek79}.

There are numerous extensions of EkVP obtained either relaxing the conditions on the function $f$ (e.g. considering functions with values in an ordered vector space) or considering spaces more general than the metric ones (uniform spaces, quasi-metric spaces, $w$-metric spaces, partial-metric spaces, etc), or both.

Extensions of EkVP and other variational  principles to quasi-metric spaces (the non-symmetric analogs of metric spaces) were obtained in \cite{kassay19}, \cite{bcs18}, \cite{cobz11}, \cite{cobz19}, \cite{karap-romag15}, and to asymmetric locally convex spaces in \cite{cobz12}.

In the present paper we show that   EkVP and its equivalents also hold  for sequentially lsc functions defined on  sequentially right $K$-complete quasi-uniform spaces. These extend some results
obtained by Hamel   \cite{hamel01}, \cite{hamel05} and Hamel and L\"ohne \cite{hamel-lohne03}  (see also Frigon \cite{frigon11})  in uniform spaces as well as by Fang \cite{fang96} in $F$-gauge spaces.  Quasi-uniform spaces are the non-symmetric version of uniform spaces (see Section \ref{S.qu-sp} for details).

Notice that some results on variational principles in quasi-uniform spaces were obtained by Fierro \cite{fierro17a}, working with another notion of completeness.

The structure of the paper is the following.  For convenience we present in Section \ref{S.qu-sp}  the basic notions and results on quasi-uniform spaces which are used throughout the paper, while Section \ref{S.Fqg-sp} contains a brief introduction to Fang quasi-gauge spaces. Ekeland Variational Principle,
Caristi-Kirk fixed point theorem, Takahashi and Arutyunov minimization principles and Oettli-Th\'era version of EkVP in $F$-quasi-gauge spaces are presented in Section \ref{S.Ek-Fqg}.  The equivalence of these principles is  proved in Section \ref{S.equiv}. The last section (Section \ref{S.fcs}) is concerned with some classes of functions, more general than the lower semicontinuous ones,  and the extension to quasi-uniform spaces of a minimization principle proved by Arutyunov and  Gel'man  \cite{arut-gel09} and Arutyunov \cite{arut15} in metric spaces.

In \cite{bcs18}, \cite{cobz11},  \cite{cobz19} and \cite{karap-romag15} it is shown that the validity of EkVP (or of its equivalent -- Caristi-Kirk fixed point theorem) implies some completeness properties of the underlying quasi-metric space. I did not succeed to prove a similar result in the quasi-uniform case.

\section{Quasi-uniform spaces}\label{S.qu-sp}

Quasi-uniform spaces are uniform spaces without the symmetry axiom ((QU5) from below). Their theory is   much  more complicated than that of uniform spaces,
mainly in what concerns completeness and compactness (see, e.g., \cite{kunzi-reily93}). A good introduction to quasi-uniform spaces is given in the book \cite{FL}. For subsequent developments
one can consult, for instance, \cite{Cobzas} or \cite{kunzi09}  and the references quoted therein.

For a set $X$, denote by   $\Delta(X)=\{(x,x) : x\in X\}$   the \emph{diagonal} of $X$
and, for $M,N\subseteq X\times X,$ let
\begin{equation*}
M\circ N =\{(x,z)\in X\times X : \exists y\in X,\; (x,y)\in
M\;\mbox{and} \; (y,z)\in N\}\;,
\end{equation*}
and
$$
M^{-1}=\{(x,y)\in X\times X: (y,x)\in M\}\,.$$

A \emph{quasi-uniformity} on  $X$ is a nonempty family  $\mathcal U$ of subsets of $X\times X$ such that

\begin{equation*}
\begin{aligned}
\mbox{(QU1)}&\quad \Delta(X)\subseteq U ;\\
\mbox{(QU2)}&\quad V\subseteq X\times X\;\mbox{ and }\; U\subseteq V \;\Rightarrow\;   V\in \mathcal{U};\\
\mbox{(QU3)}&\quad   U_1,U_2\in \mathcal{U}\;\Rightarrow\; U_1\cap U_2\in \mathcal{U};\\
\mbox{(QU4)}&\quad   \mbox{there exists } V\in \mathcal{U} \mbox{ such that }  V\circ V\subseteq U\,,
\end{aligned}
\end{equation*}
for all $U,U_1,U_2\in \mathcal U$.

If further
$$\mbox{(QU5)}\quad  \;  U^{-1}\in \mathcal{U},$$
 for all $  U\in \mathcal{U},$ then $\mathcal U$ is called a \emph{uniformity}.
\begin{remark} Conditions (QU2) and (QU3) show that $\mathcal U$ is a filter on $X\times X.$
\end{remark}

If $\mathcal U$ is a quasi-uniformity on a set  $X$, then
$$
\mathcal U^{-1}=\{U^{-1}:U\in \mathcal U\}$$
is another quasi-uniformity on $X$, called the conjugate of $\mathcal U$. Also the family of sets
$$
 \{U\cap U^{-1} :U\in\mathcal U\}$$
 is a basis for a uniformity $\mathcal U^s$ on $X$, called the associated uniformity to $\mathcal U$, or the uniformity generated by $\mathcal U$.
 $\mathcal U^s$ is the smallest uniformity on $X\times X$ containing both $\mathcal U$ and $\mathcal U^{-1}$.

 For $U\in \mathcal{U}, \, x\in X$ and $Z\subseteq X$ put
$$
U(x) =\{y\in X: (x,y)\in U\}\quad\mbox{and}\quad U[Z]=\bigcup\{U(z):
z\in Z\}\;.
$$
\begin{remark}
  Observe that
  $$U(x)=\{x\}\circ U\,.$$

  For this reason some authors denote $\{x\}\circ U$ by $(x)U$ and $U\circ\{y\}$ by $U(y)$.
\end{remark}

Any quasi-uniformity $\mathcal U$ on a set $X$ generates a topology $\tau(\mathcal U)$ on $X$ for which
the family of sets
 \begin{equation*}
\{ U(x): U\in \mathcal{U}\}\end{equation*}
is a basis of neighborhoods of the point $x\in X$ (actually it agrees with the neighborhood filter at $x$).

As a space with two topologies $\tau(\mathcal U)$ and $\tau(\mathcal U^{-1})$, a quasi-uniform space can be also viewed as a bitopological space in the sense
 of Kelly \cite{kelly63}.

The separation properties of the topology induced by a quasi-uniformity are contained in the following   proposition.

 \begin{prop}[\cite{FL}, Proposition 1.9]
 Let $(X,\mathcal U)$ be a quasi-uniform space and $\tau(\mathcal U)$ the topology generated by $\mathcal U$.
 \begin{enumerate}
 \item[\rm 1.] The topology $\tau(\mathcal U)$ is $T_0$ if and only if $\bigcap \mathcal U$ is an  order on $X$ if and only if $W\cap W^{-1}=\Delta(X)$,
 where $W=\bigcap\mathcal U.$
 \item[\rm 2.] The topology $\tau(\mathcal U)$ is $T_0$ if and only if $\tau(\mathcal U^s)$ is $T_2$.
 \item[\rm 3.] The topology $\tau(\mathcal U)$ is $T_1$ if and only if $\,\bigcap \mathcal U=\Delta(X).$
 \item[\rm 4.] The topology $\tau(\mathcal U)$ is $T_2$ if and only if $\bigcap \{U\cap U^{-1}:U\in\mathcal U\}=\Delta(X).$

\end{enumerate}\end{prop}

\begin{remark}
  As it was shown by Pervin \cite{pervin62} (see also \cite[Th. 1.1.55]{Cobzas} every topological  space $(X,\tau)$ is quasi-uniformizable, i.e. there exists a quasi-uniformity
  generating the topology $\tau$ For uniformities  this is true only for completely regular topologies, see \cite[Corollary 17, p. 188]{Kel}.
\end{remark}

\begin{remark}
  Sometimes we shall say that the quasi-uniformity $\mathcal U$ is $T_1$ (or other separation condition) understanding by this that the generated topology $\tau(\mathcal U)$ satisfies this condition.  The same will be with other topological notions: $\mathcal U$-convergence, $\mathcal U$-closedness, instead of $\tau(\mathcal U)$-convergence, $\tau(\mathcal U)$-closedness, etc.
\end{remark}
\begin{remark}
Suppose that    $ (X,\mathcal U)$ is a uniform space. Then the topology $\tau(\mathcal U)$  is $T_2$ if and only if $\bigcap\mathcal U=\Delta(X)$. Actually, in a uniform space,
the $T_0$ separation of the topology $\tau(\mathcal U)$ implies   $T_2$  and regularity (see Kelley \cite{Kel}).
\end{remark}

   A nonempty family $\mathcal{B}\subseteq \mathcal{U}$ is called a \emph{basis} for a quasi-uniformity  $\mathcal{U}$ if every $U\in \mathcal U$ contains a $B\in\mathcal B$.
A nonempty family $\mathcal{C}\subseteq \mathcal{U}$ is called a \emph{subbasis} for $\mathcal{U}$ if the family of all finite intersections of members in $\mathcal C$ is a basis for $\mathcal U$, i.e., for every $U\in\mathcal U$ there exist $n\in\mathbb{N}$ and $C_1,\dots,C_n$ in $\mathcal C$ such that  $\bigcap_{i=1}^n C_i\subseteq U.$ This means that $\mathcal B$ (respectively $\mathcal C$) is a basis
(subbasis) of the filter $\mathcal U$.

\begin{prop}\label{p1.base-qu} Let $X$ be a set.  A nonempty family $\mathcal B$ of subsets of $X$ is a basis for a quasi-uniformity $\mathcal U$ on $X$ if and only if it satisfies the following conditions
\begin{equation*}
\begin{aligned}
 {\rm (BQU1)}&\qquad  \forall B\in \mathcal{B},\; \Delta(X)\subseteq B;\\
{\rm (BQU2)}&\qquad \forall B\in \mathcal{B},\; \exists C\in \mathcal{B},\;
\mbox{ such that } \; C\circ C\subseteq B\;,\\
{\rm (BQU3)}&\qquad \forall B_1,B_2\in \mathcal{B},\; \exists B\in \mathcal{B}\;
\mbox{ such that } \; B\subseteq B_1\cap B_2\,.
\end{aligned}
\end{equation*}

In this case the quasi-uniformity $\mathcal U$ is given by
\begin{equation*}
\mathcal U=\{U\subseteq X\times X : \exists B\in\mathcal B,\, B\subseteq U\}\,.\end{equation*}
If $\mathcal B$ satisfies only (BQU1) and (BQU2), then it is a subbasis for the quasi-uniformity $\mathcal U$ given by
\begin{equation*}
\mathcal U=\{U\subseteq X\times X : \exists n\in\mathbb{N},\, \exists B_1,\dots,B_n\in\mathcal B,\, B_1\cap\dots\cap B_n\subseteq U\}\,.
\end{equation*}  \end{prop}

The following result holds.
\begin{prop}
  If $\mathcal B$ is a basis (subbasis) for a quasi-uniformity $\mathcal U$ on $X$, then, for every $x\in X,$ $\{B(x) : B\in\mathcal B\}$ is a neighborhood basis (subbasis)    at $x$ with respect to  the topology $\tau(\mathcal U).$
\end{prop}

By analogy with the uniform case one can define quasi-uniformly continuous mappings. A mapping $f$
between two quasi-uniform spaces $(X,\mathcal{U})$ and $(Y,\mathcal W)$ is
called \emph{quasi-uniformly continuous} if for every $W\in
\mathcal W$ there exists $U\in \mathcal{U}$ such that $(f(x),f(y))\in
W$ for all $(x,y)\in U.$

If $\tilde f:X\times X\to Y\times Y$ is given by
\begin{equation*}
\tilde f(x,y)=(f(x),f(y)),\, x,y\in X,\end{equation*}
then the quasi-uniform continuity of the function $f$ is equivalent to the condition
\begin{equation*}
\tilde f^{-1}(W)\in \mathcal U\,,\end{equation*}
for all $W\in\mathcal W.$

A \emph{quasi-uniform isomorphism} is a bijective quasi-uniformly continuous  function $f$ such that the inverse function $f^{-1}$ is also  quasi-uniformly continuous.

By the definition of the topology
generated by a quasi-uniformity, it is clear that a
quasi-uniformly continuous mapping is continuous with respect to
the topologies $\tau(\mathcal{U}),\, \tau(\mathcal W).$ Actually the following stronger result holds.
\begin{prop}
Let $(X,\mathcal{U}),\,(Y,\mathcal W)$ be quasi-uniform spaces and let $f:(X,\mathcal U)\to (Y,\mathcal W)$ be  quasi-uniformly continuous.  Then the following properties hold.
\begin{enumerate}\item[\rm 1.]
The function $f$  is also quasi-uniformly continuous  from  $(X,\mathcal U^{-1})$ to $(Y,\mathcal W^{-1})$ and uniformly continuous   from  $(X,\mathcal U^s)$ to $(Y,\mathcal W^s)$.
\item[\rm 2.]
The function $f$ is $\tau(\mathcal U)$-$\tau(\mathcal W)$, $\tau(\mathcal U^{-1})$-$\tau(\mathcal W^{-1})$ and $\tau(\mathcal U^s)$-$\tau(\mathcal W^s)$ continuous.
\end{enumerate}\end{prop}

A function  $d:X\times X\to \mathbb{R}_+$ is called a \emph{quasi-pseudometric} if it satisfies the following conditions:
\begin{align*}
  \mbox{(QM1)}&\quad d(x,x)=0;\\
  \mbox{(QM1)}&\quad d(x,z)\le d(x,y)+d(y,z)\,,
  \end{align*}
  for all $x,y,z\in X$.

  If further
  $$
  \mbox{(QM3)}\quad d(x,y)=0=d(y,x)\;\Rightarrow\; x=y\,,
  $$
  then $d$ is called a \emph{quasi-metric}.  The mapping $\bar d(x,y)=d(y,x),\, x,y\in X,$ is also a quasi-pseudometric on $X$, called the conjugate of $d$ and
   $d^s(x,y)=\max\{d(x,y),d(y,x)\}$ is a pseudometric, which  is a metric if and only if $d$ is a quasi-metric.

   A quasi-pseudometric $d$  induces a quasi-uniformity $\mathcal U_d$ on $X$ having a basis formed of the sets
   \begin{equation}\label{def.Vd}
   V_{d,\varepsilon}=\{(x,y)\in X\times X: d(x,y)<\varepsilon\},\; \varepsilon >0\,.\end{equation}

     A typical example is that of the space $\mathbb{R}$ with the quasi-pseudometric induced by the asymmetric  seminorm $u:\mathbb{R}\to\mathbb{R}_+,\, u(\alpha)=\alpha^+,\,\alpha\in \mathbb{R},$
and the associated quasi-pseudometric  $d_u(\alpha,\beta)=u(\beta-\alpha),\, \alpha,\beta\in\mathbb{R}.$  The conjugate asymmetric seminorm is $\bar u(\alpha)=u(-\alpha)=\alpha^-$, with $\overline{d_u}=d_{\bar u}$, and the associated norm, the absolute value $u^s(\alpha)=|\alpha|.$  We shall denote the quasi-uniformity associated to $d_u$ by $\mathcal R_u.$

Since, for every $\alpha,\beta\in\mathbb{R},$
$$
\beta-\alpha<\varepsilon\iff u(\beta-\alpha)<\varepsilon\,,$$
the quasi-uniform continuity of a function $f$ from a quasi-uniform space $(X,\mathcal U)$ to $(\mathbb{R},\mathcal R_u)$ can be characterized in the following way.
\begin{prop} A function $f:(X,\mathcal U)\to (\mathbb{R},\mathcal R_u)$ is quasi-uniformly continuous if and only if for every $\varepsilon>0$ there exists $U\in\mathcal U$ such that
\begin{equation*}
u(f(y)-f(x))<\varepsilon\;\mbox{ for all }\; (x,y)\in U\,,\end{equation*}
or, equivalently,
\begin{equation*}
f(y)-f(x)<\varepsilon\;\mbox{ for all }\; (x,y)\in U\,.\end{equation*}
\end{prop}

One says also that   $f$ is $(\mathcal U,u)$-quasi-uniformly continuous.

 Let  $(X,\mathcal U)$ be  a quasi-uniform space.   One can also define a product quasi-uniformity  $\,\mathcal U\times \mathcal U$ on  $X\times X$ through the basis  formed by  the family  of subsets of $X^4$
\begin{equation*}\label{def.prod-qu}
\{((u,v),(x,y))\in X^2\times X^2: (u,x), (v,y)\in U\},\; U\in\mathcal U.\end{equation*}

It is the smallest quasi-uniformity on $X\times X$ making the canonical projections $p_1,p_2:(X\times X,\mathcal U\times \mathcal U)\to (X,\mathcal U)$ quasi-uniformly continuous.

The product quasi-uniformities $\mathcal U\times\mathcal U^{-1}$ and $\,\mathcal U^{-1}\times\mathcal U$ are defined analogously.

   The quasi-uniform continuity of a quasi-pseudometric $d$ can be characterized in the following way.  Actually it is an adaptation to the asymmetric case of Theorem 11, Ch. 6, from \cite{Kel}.
   \begin{prop}\label{p.quc-qm}
Let $(X,\mathcal{U})$ be a quasi-uniform space and $d$ a quasi-pseudometric on $X$.\begin{enumerate}\item[\rm 1.]
 The quasi-pseudometric $d$ is $(\mathcal U^{-1}\times\mathcal U,u)$-quasi-uniformly continuous if and only if
$V_{d,\varepsilon}\in\mathcal U$ for all $\varepsilon>0$.
 \item[\rm 2.] If $d$ is $(\mathcal U^{-1}\times\mathcal U,u)$-quasi-uniformly continuous, then for every fixed $x_0\in X$ the mapping $d(x_0,\cdot):X\to\mathbb{R}$ is $\tau(\mathcal U)$-usc and $d(\cdot,x_0):X\to\mathbb{R}$ is $\tau(\mathcal U)$-lsc.
\end{enumerate} \end{prop} \begin{proof} 1.\;   A basis for the product quasi-uniformity $\mathcal U^{-1}\times \mathcal U$ is formed by the sets
\begin{equation}\label{def2.prod-qu}
  W_U:=\{((u,v),(x,y))\in X^2\times X^2: (x,u),(v,y)\in U\},\; U\in\mathcal U.\end{equation}

  Suppose that $d$ is quasi-uniformly continuous in the sense mentioned in proposition and let $\varepsilon>0.$ Then there exists $U\in \mathcal U$ such that
  $$
  ((u,v),(x,y))\in  W_U\; \Rightarrow\; d(x,y)-d(u,v) <\varepsilon\,,$$
  where $ W_U$ corresponds to $U$ by the  equality \eqref{def2.prod-qu}.

  If $(x,y)\in U,$ then $((x,x),(x,y))\in W_U$, so that $d(x,y)=d(x,y)-d(x,x)<\varepsilon$, that is $(x,y)\in V_{d,\varepsilon}.$
  This shows that $U\subseteq V_{d,\varepsilon}$ and so $V_{d,\varepsilon}\in \mathcal U$.

  Conversely, suppose that  $V_{d,\varepsilon}\in \mathcal U$ for every $\varepsilon>0.$  For given  $\varepsilon>0$ let  $U:=V_{d,\varepsilon}\in\mathcal U$ and
\begin{align*}
   W_U&=\{((u,v),(x,y))\in X^2\times X^2 : (x,u),(v,y)\in U\}\\
   &=\{((u,v),(x,y))\in X^2\times X^2 : d(x,u)<\varepsilon\,\mbox{ and }\, d(v,y)<\varepsilon\}\,.
\end{align*}

For every $((u,v),(x,y))\in W_U$,
\begin{align*}
  d(x,y)&\le d(x,u)+d(u,y)<\varepsilon+d(u,y)\\
  d(u,y)&\le d(u,v)+d(v,y)<d(u,v)+\varepsilon\,,
\end{align*}
which by addition yield
$$
d(x,y)<d(u,v)+2\varepsilon\,,$$
showing that $d$ is $(\mathcal U^{-1}\times\mathcal U,u)$-quasi-uniformly continuous

  2.\; Let  $\varepsilon>0$. By the $(\mathcal U^{-1}\times\mathcal U,u)$-quasi-uniform continuity of $d$ there exists $U\in\mathcal U$ such that
   \begin{equation}\label{eq1.lsc-qsm}
    d(x,y)<d(u,v)+\varepsilon\,,\end{equation}
    for all $((u,v),(x,y))\in  W_U$, where $ W_U$ is defined by \eqref{def2.prod-qu}.

    Let $x\in X$  and $y\in U(x)$. Then $(x_0,x_0),(x,y)\in U$ implies   $((x_0,x),(x_0,y))\in  W_U$ and, by \eqref{eq1.lsc-qsm},
     \begin{equation*}
    d(x_0,y)<d(x_0,x)+\varepsilon\,,\end{equation*}
    proving the $\tau(\mathcal U)$-usc  of $d(x_0,\cdot)$ at $x$.

    Now, $(x,y),(x_0,x_0)\in U$ implies $((y,x_0),(x,x_0))\in  W_U$ so that, by \eqref{eq1.lsc-qsm},
    $$
    d(x,x_0)<d(y,x_0)+\varepsilon\,,$$
    or equivalently,
    $$
    d(y,x_0)>d(x,x_0)-\varepsilon\,,$$
     proving the $\tau(\mathcal U)$-lsc  of $d(\cdot,x_0)$ at $x$.
\end{proof}

  The construction of a quasi-uniformity from a quasi-pseudometric can be extended to a family $P$ of quasi-pseudometrics on a set $X$. In this case, the family of sets
  $$
  \mathcal B_P=\{V_{d,\varepsilon} : d\in P,\, \varepsilon>0\}\,,$$
  where $V_{d,\varepsilon}$ is given by \eqref{def.Vd}, is a subbasis of a quasi-uniformity $\mathcal U_P$ on $X$.
   \begin{prop} \label{p.qu-qms} Let $P$ be a family of quasi-pseudometrics on a set $X$.
   The quasi-uniformity $\mathcal U_P$  is the smallest quasi-uniformity $\mathcal U$ on $X$ making all the quasi-pseudometrics in $P$, \;$(\mathcal U^{-1}\times\mathcal U,u)$-quasi-uniformly continuous.
    \end{prop}

  As was shown by  Reilly  \cite{reily70a} (see also  \cite{reily70b}, \cite{reily73})    the following result, a nonsymmetric analog of a well-known result in uniform spaces (see Kelley \cite{Kel}),  holds.

  \begin{theo}
  Let $X$ be a set. Any quasi-uniformity $\mathcal U$ on  $X$   is generated by a family  of $(\mathcal U^{-1}\times\mathcal U,u)$-quasi-uni\-formly continuous quasi-pseudometrics in the way described above.\end{theo}

We also mention  the following result.

\begin{prop}
Let $P$ be a family of quasi-pseudometrics on a set $X$ and $\mathcal U_P$ the generated quasi-uniformity.
\begin{enumerate}\item[\rm 1.]  The family   $Q$  of all $(\mathcal U^{-1}_P\times\mathcal U_P,u)$-quasi-uniformly continuous quasi-pseudometrics on $X$ generates the same quasi-uniformity as $P$, that is, $\mathcal U_Q=\mathcal U_P$.
\item[\rm 2.] The following equality holds
\begin{equation*} \mathcal U_P^{-1}=\mathcal U_{\overline P}\,,\end{equation*}
where  $\overline P$ is the set of quasi-pseudometrics conjugate to those in $P$, i.e.  \begin{equation*} \overline P=\{\bar p : p\in P\}\,.\end{equation*}
\end{enumerate}    \end{prop}

    The second assertion from the above proposition follows from the equality
    $$
    V^{-1}_{d,\varepsilon}=  V_{\bar d,\varepsilon}\,.$$

    The convergence of nets and the separation properties of the topology generated by a family of quasi-pseudometrics can be expressed in terms of these quasi-pseudometrics.
    \begin{prop}\label{p2.sep-P} Let $P$ be a family of quasi-pseudometrics on a set $X$ and $\mathcal U=\mathcal U_P$ the quasi-uniformity on $X$ it generates.
    \begin{enumerate}\item[\rm 1.] A net $(x_i:i\in I)$ in $X$ is $\tau(\mathcal U)$-convergent to $x\in X$ if and only if
    $$
    \forall d\in P,\quad d(x,x_i)\to 0\,.$$
    \item[\rm 2.] For every $d\in P$ the function $d(\cdot,y)$ is $\mathcal{U}$-usc and $\mathcal{U}^{-1}$-lsc for every $y\in X$, while
    $d(x,\cdot)$ is $\mathcal{U}$-lsc and $\mathcal{U}^{-1}$-usc for every $x\in X$.
  \item[\rm 3.] The topology $\tau(\mathcal U)$ is $T_0$ if and only if  for every  $x,y\in X$ with $x\ne y$   there exists  $d\in P$ such that  $d(x,y)>0$ or $d(y,x)>0$.
    \item[\rm 4.] The topology $\tau(\mathcal U)$ is $T_1$ if and only if   for every  $x,y\in X$ with $x\ne y$   there exists  $d\in P$ such that $d(x,y)>0$.
    \end{enumerate}\end{prop}
    \begin{proof} We present only the proof of the $\mathcal{U}$-lsc of the function $d(\cdot,y)$ which will be used later. Let $(x_i:i\in I)$ be a net in $X$ that is  $\mathcal{U}$-convergent to some $x\in X$. By assertion 1 this implies
    $$\lim_id(x,x_i)=0\,,$$
    for every $d\in P$.

     But then
    $$
    d(x,y)\le d(x,x_i)+d(x_i,y)\,,$$
    for all $i\in I$,  so that
    $$
    d(x,y)\le \liminf_id(x_i,y)\,.$$
    This shows that $d(\cdot,y)$ is $\mathcal{U}$-lsc.

    \end{proof}

     \begin{remark} Kelley \cite[p. 189]{Kel} calls the family of all $(\mathcal U\times\mathcal U)$-uniformly continuous pseudometrics on a uniform space $(X,\mathcal U)$  the ``gage"  of the uniform space $X$. Reilly \cite{reily73} calls a family of quasi-pseudometrics on a set $X$ a quasi-gauge and  a topological space $(X,\tau)$ whose topology is generated by a family $P$ of quasi-pseudometrics (i.e. $\tau=\tau(\mathcal U_P)$),  a quasi-gauge space.

      We shall follow   Reilly \cite{reily73} and call a family of quasi-pseudometrics on a set $X$  a \emph{quasi-gauge}.
        \end{remark}

    There are several notions of Cauchy net and  Cauchy filter in quasi-uniform spaces and the corresponding notions of completeness (see \cite{Cobzas} or \cite{kunzi09}), similar to those in quasi-metric spaces (see \cite{reily-subram82}).  All these notions agree with the usual notion of completeness in the case of uniform spaces.  Since the  results included in this paper will be concerned only with the sequential notions, we restrict the presentation to this case and to   sequential  right and left $K$-completeness.

    \begin{defi}
    Let $(X,\mathcal U)$ be a quasi-uniform space.
     \begin{itemize}\item
  a sequence $(x_n)$ in $X$  is  called \emph{left} $K$-\emph{Cauchy} (or \emph{left} $\mathcal U$-$K$-\emph{Cauchy}  if more precision is needed)  if for every $U\in \mathcal U$ there exists $n_U\in\mathbb{N}$ such that
 \begin{equation}\label{def1.lKC} \forall n\ge n_U,\; \forall k\in\mathbb{N},\quad (x_n,x_{n+k})\in U\,;\end{equation}
 \item the sequence $(x_n)$  is called  \emph{right} $K$-\emph{Cauchy} (or \emph{right} $\mathcal U$-$K$-\emph{Cauchy}) if it satisfies \eqref{def1.lKC} with  $(x_{n+k},x_n)$ instead of $(x_n,x_{n+k})$;
     \item one says that the quasi-uniform space $(X,\mathcal U)$ is   \emph{sequentially left} $K$-\emph{complete} (or \emph{sequentially left} $\mathcal U$-$K$-\emph{complete})   if every left $K$-Cauchy sequence in $X$ is $\tau(\mathcal U)$-convergent to some $x\in X$;
     \item the  notion of \emph{sequential right} $K$-\emph{completeness} is defined accordingly.
     \end{itemize}\end{defi}

     We have again characterizations in terms of the quasi-pseudometrics generating the quasi-uniformity.
\begin{prop}
  Suppose that  the quasi-uniformity $\mathcal U$ is generated by a family $P$ of quasi-pseudometrics. Then $(x_n)$ is left  $K$-Cauchy (right $K$-Cauchy) if and only if for every $d\in P$ and $\varepsilon>0$ there exists $n_0=n_{d,\varepsilon}\in \mathbb{N}$ such that
 \begin{equation*}
 \forall n\ge n_0,\; \forall k\in\mathbb{N},\quad d(x_n,x_{n+k})<\varepsilon\,,\end{equation*}
 or
 \begin{equation*}
 \forall n\ge n_0,\; \forall k\in\mathbb{N},\quad d(x_{n+k},x_{n})<\varepsilon\,,\end{equation*}
 respectively.
 \end{prop}

    \section{Fang's quasi-gauge spaces}\label{S.Fqg-sp}

    Fang \cite{fang96} considered a slightly more general notion of gauge, called $F$-gauge.  We extend it to the asymmetric case.
    \begin{defi}\label{def.Fqg} Let $X$ be a set. A family  $D$   of functions $d:X\times X\to \mathbb{R}$ is called an $F$-\emph{quasi-gauge} if it satisfies the following conditions:
    \begin{align*}
  {\rm (QF1)}& \quad \mbox{the set } D \mbox{ is directed upward};\\
{\rm (QF2)}& \quad \forall d\in D,\;\forall x,y\in X, \; d(x,y)\ge 0\;\wedge d(x,x)=0;\\
 {\rm (QF3)}& \quad \forall d\in D,\; \exists d'\ge d \mbox{ in } D \mbox{ s.t. }\; (\forall x,y,z\in X, \; d(x,y)\le d'(x,z)+d'(z,y))\,.
  \end{align*}

  If further
  $$\mbox{(QF4)}\quad \forall d\in D,\; \forall x,y\in X, \; d(x,y)=d(y,x)\,,$$
  then $D$ is called an $F$-\emph{gauge}.\end{defi}

  Fang \cite{fang96}  defined an $F$-space as a topological space $(X,\tau)$ whose topology is defined by an $F$-gauge in the sense that
  a neighborhood basis at every point $x\in X$ is formed by the sets $$U_d(x,\varepsilon):=\{y\in X: d(x,y)<\varepsilon\},\; d\in D,\, \varepsilon>0\,. $$

  Later Hamel and L\"ohne \cite{hamel-lohne03} have shown that, similarly to the case of gauges of pseudometrics,  an $F$-gauge generates a uniformity  $\mathcal U_D$ on $X$ having as basis the sets
 $$V_{d,\varepsilon}:=\{(x,y)\in X\times X: d(x,y)<\varepsilon\},\; d\in D,\, \varepsilon>0\,.$$

 Since $U_d(x,\varepsilon)=V_{d,\varepsilon}(x),$ the topological space   $(X,\tau(\mathcal U_D))$ is an $F$-space in Fang's sense.

 We show that an $F$-quasi-gauge generates a quasi-uniformity on $X$.
 \begin{prop}
  Let $X$ be a set and $D$ an $F$-quasi-gauge on $X$. Then the family of sets
 $$
 \mathcal B_D:=\{V_{d,\varepsilon}:d\in D,\,\varepsilon>0\}$$
 is a basis for a quasi-uniformity $\mathcal U_D$ on $X$ given by
 $$
 \mathcal U_D=\{U\subseteq X\times X: \exists d\in D,\,\exists\varepsilon>0,\; V_{d,\varepsilon}\subseteq U\}\,.$$

 The conjugate quasi-uniformity $\mathcal U^{-1}_D$  has as basis the family of sets
 $$
 \mathcal B_{\overline D}:=\{V_{\bar d,\varepsilon}:d\in D,\,\varepsilon>0\}\,,$$
 that is $\mathcal U^{-1}_D=\mathcal U_{\overline D}\,,$
 where
 $$
 \overline D=\{\bar d:d\in D\}\,.$$
 \end{prop}\begin{proof}
   We have to show that $\mathcal B_D$ satisfies the conditions (BQU1)--(BQU3) from Proposition \ref{p1.base-qu}. Since $d(x,x)=0$ it follows $\Delta(X)\subseteq V_{d,\varepsilon}.$ Also
   if $V_{d_1,\varepsilon_1}, V_{d_2,\varepsilon_2}\in\mathcal B_D,$ then $V_{d,\varepsilon}\subseteq V_{d_1,\varepsilon_1}\cap V_{d_2,\varepsilon_2}$, where $d\ge d_1,d_2$ ($D$ is directed) and $\varepsilon:=\min\{\varepsilon_1,\varepsilon_2\}.$

   For $V_{d,\varepsilon}\in\mathcal B_D$ let  $d'$ be given by (QF3). Then $V_{d',\varepsilon/2}\circ V_{d',\varepsilon/2}\subseteq V_{d,\varepsilon}.$

   The assertion concerning the conjugate quasi-uniformity follows from the equality
   $$
   V_{d,\varepsilon}^{-1}=V_{\bar d,\varepsilon}\,,$$
   valid for all $d\in D$ and $\varepsilon>0$.
 \end{proof}

 Similar to Proposition \ref{p2.sep-P}, the topological properties of a quasi-uniform space can be characterized in  terms of an $F$-quasi-gauge generating the quasi-uniformity.
 \begin{prop}\label{p2.top-Fqg} Let $D$ be an $F$-quasi-gauge on a set $X$ and $\mathcal U=\mathcal U_D$ the generated quasi-uniformity.
 \begin{enumerate}\item[\rm 1.] A net $(x_i:i\in I\}$ in $X$ is $\tau(\mathcal U)$-convergent to $x\in X$ if and only if
    $$
    \forall d\in D,\quad d(x,x_i)\to 0\,.$$
        \item[\rm 2.] Let $(x_i:i\in I)$ be a net in $X$ that is $\mathcal{U}$-convergent to some $x\in X$. Then  for every $y\in X$ and every  $d\in D$
       $$d(x,y)\le \liminf_id'(x_i,y)\,,$$
       where $d$ and $d'$ are related by (QF3).
        \item[\rm 3.] The topology $\tau(\mathcal U)$ is $T_0$ if and only if  for every  $x,y\in X$ with $x\ne y$   there exists  $d\in D$ such that  $d(x,y)>0$ or $d(y,x)>0$.
    \item[\rm 4.] The topology $\tau(\mathcal U)$ is $T_1$ if and only if   for every  $x,y\in X$ with $x\ne y$   there exists  $d\in D$ such that $d(x,y)>0$.
    \end{enumerate}\end{prop}
    \begin{proof} All these properties follow from the fact that the sets $V_{d,\varepsilon},\ d\in D,\, \varepsilon>0$, form a basis for the quasi-uniformity $\mathcal U_D$.

    1.\; This follows from the equivalences
    \begin{align*}
      x_i\xrightarrow{\tau(\mathcal U)}x&\iff \forall U\in\mathcal U,\;  \exists i_U\in I\; \mbox{ s.t. }\; (x,x_i)\in  U \;\mbox{for all }\; i\ge i_U\\
      &\iff \forall  d\in D,\forall \varepsilon>0,\,\exists i_0=i_0(d,\varepsilon)\in I \; \mbox{ s.t. }\; d(x,x_i)< \varepsilon \;\mbox{for all }\; i\ge i_0\\
      &\iff \forall d\in D,\;  \lim_id(x,x_i)=0 \,.
    \end{align*}

    2.\; This follows from  1 and  the inequality
    $$
    d(x,y)\le d'(x,x_i)+d'(x_i,y)\,$$
    (see the proof of the corresponding item in Proposition \ref{p2.sep-P}).

 3.\; In this case
    \begin{align*}
      \tau(\mathcal U)\;\mbox{ is }\; T_0&\iff \forall x,y\in X,\; (x\ne y\;\Rightarrow\; \exists V_{d,\varepsilon}\in\mathcal B_D.\; (x,y)\notin  V_{d,\varepsilon}\vee (y,x)\notin V_{d,\varepsilon})\\&\iff \forall x,y\in X,\; (x\ne y\;\Rightarrow\; \exists d\in D,\,\exists\varepsilon>0, \;  d(x,y)\ge \varepsilon \vee d(y,x)\ge\varepsilon)\\
      &\iff \forall x,y\in X,\; (x\ne y\;\Rightarrow\; \exists d\in D, \; d(x,y)>0 \vee d(y,x)>0)\,.
    \end{align*}

    The proof of 4 is similar to that of 3, replacing $\vee$ with $\wedge$.
    \end{proof}

    The Cauchy properties of  sequences in $(X,\mathcal U_D)$   can also be expressed in terms of the $F$-quasi-gauge $D$.
     \begin{prop}
     Let $D$ be an $F$-quasi-gauge on a set $X$ and $\mathcal U=\mathcal U_D$ the generated quasi-uniformity.

  A sequence   $(x_n)_{n\in\mathbb{N}}$ in $X$ is left $K$-Cauchy  (right $K$-Cauchy) if and only if
    for every $d\in D$ and $\varepsilon>0$ there exists $n_0=n_0(d,\varepsilon)\in \mathbb{N}$  such that for every $n\ge n_0 $  and every $k\in\mathbb{N}$,
    \begin{equation*}
    \quad  d(x_n,x_{n+k})<\varepsilon\,,\end{equation*}
respectively,
\begin{equation*}
\quad d(x_{n+k},x_n)<\varepsilon\,.\end{equation*}
\end{prop}

\section{Ekeland variational principle in $F$-quasi-gauge spaces}\label{S.Ek-Fqg}

Fang \cite{fang96} introduced the class of $F$-gauge spaces and proved some results on Ekeland's variational principle and   its equivalents in these spaces.
Hamel and L\"ohne \cite{hamel-lohne03} have shown that an $F$-gauge space is a uniform space and  Hamel \cite{hamel01} proved that, similarly to the case of uniform spaces,  Fang's results can be obtained by an application of Brezis-Browder principle on ordered sets.
Following Hamel and L\"ohne's ideas (see \cite{hamel01},  \cite{hamel05}, \cite{hamel-lohne03}), we show  that  Brezis-Browder maximality principle can   be further used to extend  Fang's \cite{fang96}  results from $F$-gauge spaces   to $F$-quasi-gauge spaces (introduced in Section \ref{S.Fqg-sp}).

  We start with this principle.
If $\pq$ is a partial order on  set $Z$, then we
  use the notation $x\prec y$ to designate the situation $ x\pq y$
 and $x\neq y.$

 \begin{theo}[H. Br\'ezis and F. Browder, 1976,  \cite{brez-browd76}, Corollary 1]
  Let $(Z,\pq)$ be a partially ordered set and $\varphi:X\to\mathbb{R}$ a function on $X$. Suppose that
  the following conditions are satisfied:
  \begin{enumerate}
    \item[{\rm (a)}] the function $\varphi$ is  bounded above and strictly increasing, i.e., \;
  $x\prec y\;\Rightarrow \; \varphi(x) < \varphi(y) $;
 \item[{\rm (b)}] for any increasing sequence  $x_1 \pq x_2\pq \dots$  in $Z$  there
 exists $y\in Z$ such that  $x_n\pq y,\,$ for all $ n\in \mathbb{N}.$
  \end{enumerate}
 Then for each $x\in Z$ there exists    a   maximal element $\overline x$ in $Z$ such that $x\pq \overline x.$ %
 \end{theo}

 Reversing the order, i.e. considering the order
 $$
 x\le y\iff y\pq x,\; x,y\in Z\,,$$
 one obtains conditions for the existence  of minimal elements.
 \begin{theo}[Brezis-Browder  minimality principle]\label{t.BB-Fqg}
  Let $(Z,\le)$ be a partially ordered set and $\varphi:X\to\mathbb{R}$ a function on $X$. Suppose that
  the following conditions are satisfied:
  \begin{enumerate}
    \item[{\rm (a)}] the function $\varphi$ is  bounded below  and  strictly decreasing, i.e., \;
  $x\le y $ and $x\ne y \;\Rightarrow \; \varphi(x) > \varphi(y) $;
 \item[{\rm (b)}] for any decreasing sequence  $x_1\ge x_2\ge \dots $  in $Z$  there
 exists $y\in Z$ such that  $y\le x_n,\,$ for all \,$ n\in \mathbb{N}.$
  \end{enumerate}
 Then for each $x\in Z$ there exists    a   minimal element $\underline x$ in $Z$ such that $\underline x\le   x.$
 \end{theo}

 \begin{remark} A slightly more general form of Brezis-Browder maximality principle was obtained in \cite{carja-urs93} (see also \cite[Theorem 2.1.1]{CNV}). Let $(Z,\le)$ be a preordered set and $\varphi:Z\to \mathbb{R}\cup\{\infty\}$ a function defined on $Z$. An element $z\in Z$ is called $\varphi$-maximal if $z\le x$ implies $\varphi(x)=\varphi(z)$, for every  $x\in Z$.
Suppose that
 \begin{itemize}
   \item[\rm (i)] every increasing sequence in $Z$ is bounded  from above;
   \item[\rm(ii)] the function $\varphi$ is increasing.
 \end{itemize}

 Then for every $x\in Z$ there exists a $\varphi$-maximal element $z\in Z$ such that $x\le z$.

 The proof is based on an axiom, weaker than the axiom of choice, called the Axiom of Dependent Choice:\smallskip

 (DC) \quad If $\mathcal{R}$  is a relation on a set  $Y$ such that $\{y\in Y: x\mathcal{R} y\}\ne\emptyset$ for every $x\in Y$, then for every $y\in Y$ there exists a sequence
$ (x_n)_{n=0}^\infty$ such that $x_0=y$ and  $x_n\mathcal{R}x_{n+1} $ for all $n\in\mathbb{N}_0$.\smallskip

Turinici \cite{turin11} proved that this form of the Brezis-Browder maximality principle is equivalent to the following version of the Ekeland variational principle:
\smallskip

Let $(Z,\le)$ be  a preordered set and let $\varphi:Z\to [0,\infty]$ be a function on $Z$. Suppose that the conditions (i) and (ii) from above are satisfied and that\smallskip

(iii)\; for every $x\in Z$ and every $\varepsilon >0$ there exists $y\in Z$ such that $x\le y$ and $\varphi(y)\le\varepsilon$.\smallskip

Then for every $x\in Z$ there exists $z\in Z$ such that $x\le z$ and $\varphi(z)=0$  (it follows that $z$ is $\varphi$-maximal).\smallskip

As it is shown in \cite{turin11}, in order to obtain the initial version proved by Ekeland some further conditions have to be imposed. That paper also contains a thorough analysis of this circle of ideas, including the equivalence of these principles to
(DC).
 \end{remark}

 We shall apply Brezis-Browder principle to an order given by a function on an $F$-quasi-gauge space.

\begin{prop}\label{p1.ord-Fqg} Let $D$ be an $F$-quasi-gauge on a set $X$, $\mathcal U=\mathcal U_D$ the generated quasi-uniformity on $X$ and let $\varphi:X\to\mathbb{R}$  be a function on $X$.  The  relation $\le_\varphi$ on $X$ defined  by
\begin{equation}\label{def1.ord-Fqg}
x\le_\varphi y\iff \forall d\in D,\; \varphi(y)+d(y,x)\le \varphi(x)\,,\end{equation}
 for all $x,y\in X$, satisfies the conditions:
 \begin{itemize}\item[\rm (a)]
 the relation $\le_\varphi$ given by \eqref{def1.ord-Fqg} is a preorder on $X$  and $\varphi$ is  decreasing with respect to $\le_\varphi\,$; if the topology $\tau(\mathcal U)$ is $T_0$, then $\le_\varphi$ is an order;
 \item[\rm (b)]   if the topology $\tau(\mathcal U)$ is $T_1$, then $\varphi$ is strictly decreasing with respect to $\le_\varphi$, i.e.
 $$
 (x\le_\varphi y\;\wedge\; x\ne y)\;\Rightarrow\; \varphi(x)>\varphi(y)\,.$$
\end{itemize}
 \end{prop}\begin{proof}
It is immediate that the relation $\le_\varphi$ given by \eqref{def1.ord-Fqg} is a preorder, i.e. a reflexive and transitive relation on $X$.
Since $x\le_\varphi y$ is equivalent to
$$
0\le d(y,x)\le\varphi(y)-\varphi(x)\,,$$
it follows that $\varphi$ is decreasing with respect to $\le_\varphi$.

 Suppose that $x\le_\varphi y$ and $y\le_\varphi x$. Then
 $$
 d(y,x)\le \varphi(x)-\varphi(y)\quad\mbox{and}\quad d(x,y)\le \varphi(y)-\varphi(x)\,,$$
 which by addition yield $ 0\le d(x,y)+d(y,x)\le 0$, for all $d\in D$. It follows
 \begin{equation}\label{eq2.ord-Fqg} d(x,y)=0=d(y,x)\,,\end{equation}
 for all $d\in D.$

Assertion 2 of Proposition \ref{p2.top-Fqg} can be rewritten in the reversed form: the topology $\tau(\mathcal U_D)$ is $T_0$ if and only if
$$
\forall d\in D,\; (d(x,y)=0\;\wedge\; d(y,x)=0)\;\Rightarrow\; x=y\,.$$

Consequently, \eqref{eq2.ord-Fqg} implies $x=y$ and so (a) holds.

(b)\; Let  $x\ne y$ in $X$ with  $\, x\le_\varphi y$. If $\tau(\mathcal U_P)$ is $T_1$, then, by Proposition \ref{p2.top-Fqg}, item 3, there exists $d_0\in D$ such that $d_0(y,x)>0,$ implying
$$
0<d_0(y,x)\le\varphi(x)-\varphi(y)\,,$$
that is, $\varphi(x)<\varphi(y).$
 \end{proof}

 A typical situation when the Brezis-Browder principle applies is
 contained in the following proposition.

  \begin{prop}\label{c.BB-Fqg} Let $(X,\mathcal U)$ be a quasi-uniform space whose quasi-uniformity is generated by an $F$-quasi-gauge $D$, i.e. $\mathcal U=\mathcal U_D$.

For a function $\varphi:X\to\mathbb{R}$  let  the relation $\le_\varphi$  be given by \eqref{def1.ord-Fqg}.
   Suppose that   the quasi-uniform space $(X,\mathcal U)$ is sequentially right  $K$-complete, the topology $\tau(\mathcal U)$ is $T_1$ and that the function $\varphi$ is bounded below. If the set
 $$S_-(x):=\{y\in X: y\le_\varphi x\}\,,$$
 is sequentially $\mathcal{U}$-closed for every $x\in X$,
then  every element of $X$ is minorized by a minimal element in $(X,\le_\varphi)$.
   \end{prop}\begin{proof}
   1.\; We intend to apply Theorem \ref{t.BB-Fqg},  so we have to check the conditions (a), (b) from this theorem.

   By Proposition \ref{p1.ord-Fqg}, $\le_\varphi$ is an order on $X$ and $\varphi$ is strictly decreasing with respect to $\le_\varphi$, so that (a) holds.

   To prove (b), suppose that  $x_1\ge_\varphi x_2\ge_\varphi\dots$ is a decreasing sequence in $X$. Then $\alpha\le\varphi(x_{n+1})\le\varphi(x_{n}),\, n\in\mathbb{N},$ where $\alpha:=\inf\varphi(X)$.
   It follows that the sequence $(\varphi(x_n))_{n\in\mathbb{N}}$ is convergent, and so Cauchy.

   Since   $x_{n+k}\le_\varphi x_n$, it follows that, for every $d\in D$,
\begin{equation*}
   d(x_{n+k},x_n)\le \varphi(x_{n})-\varphi(x_{n+k})\,,
   \end{equation*}
  for all $n,k\in\mathbb{N}$,  which implies that $(x_n)$ is right $\mathcal{U}$-$K$-Cauchy,
 so that,  by the completeness hypothesis  on $X$, there exists  $y\in X$ such that
   $x_n\xrightarrow{\tau(\mathcal U)}y$.

For every $n\in\mathbb{N}$, the inequality $x_{n+k}\le_\varphi x_n$ implies $x_{n+k}\in S_-(x_n)$ for all $k\in\mathbb{N}$. By the $\mathcal{U}$-closedness of the
set $S_-(x_n)$, \, $y\in S_-(x_n)$, that is $y\le_\varphi x_n$, for all $n\in\mathbb{N}$.
\end{proof}

     \begin{remark}\label{re.lsc-clos}
       If the function $\varphi$ is sequentially $ \mathcal{U}$-lsc, then the set $S_-(x)$ is sequentially $\tau(\mathcal{U})$-closed for every $x\in X$.
            \end{remark}

Indeed, let $(y_n)$ be a sequence in $S_-(x)$ that is  $ \mathcal{U}$-convergent to some $y\in X$. Then,  for every $\delta\in D$,
\begin{equation}\label{eq0.lsc-Fqg}\begin{aligned}
  {\rm (i)}&\;\;\lim_{n\to\infty}\delta(y,y_n)=0\;\mbox{ and}\\
  {\rm (ii)}&\;\;f(y_n)+\delta(y_n,x)\le f(x)\;\mbox{ for all  }\; n\in\mathbb{N}.
\end{aligned}\end{equation}

Given $d \in D$  choose $d'\in D$ according to condition (QF3) from Definition \ref{def.Fqg}.  By Proposition \ref{p2.top-Fqg}, item 2, the $\mathcal{U}$-lsc of the function $f$ and the inequality (ii) from \eqref{eq0.lsc-Fqg},
$$
f(y)+d(y,x)\le\liminf_n[f(y_n)+d'(y_n,x)]\le f(x)\,.$$

This shows that $y\in S_-(x)$.

Let $D$ be an $F$-quasi-gauge on  a set $X$.
For a function $\varphi:X\to\mathbb{R}\cup\{\infty\}$ and $x\in X$ consider the   set
\begin{equation}\label{def.S(x)}
S_\varphi(x)=\{y\in X:\varphi(y)+d(y,x)\le\varphi(x)\, \mbox{ for all }\, d\in D\}\,.\end{equation}

      Appealing to Proposition  \ref{c.BB-Fqg} one can prove the following variant of Ekeland Variational Principle in $F$-quasi-gauge spaces.

\begin{theo}\label{t1.EkVP-Fqg} Let $(X,\mathcal U)$ be a quasi-uniform space whose quasi-uniformity is generated by an $F$-quasi-gauge $D$, (i.e. $\mathcal U=\mathcal U_D$)
 and $f:X\to \mathbb{R}\cup\{\infty\}$.
Suppose that  $(X,\mathcal U)$ is   sequentially right $K$-complete, the topology $\tau(\mathcal U)$ is $T_1$  and $f  $ is  proper, bounded below and
\begin{itemize}\item[\rm(a)] the set $S_f(x)$
is sequentially $\mathcal U$-closed for every $x\in X$.\end{itemize}

 Then, for every  $x_0\in \dom f$   there exists $z\in X$ such that
\begin{align*}
  &{\rm (i)}\;\; \forall d\in D,\quad f(z)+ d(z,x_0)\le f(x_0);\\
      &{\rm (ii)}\;\; \forall x\in X\setminus \{z\},\; \exists  d_x\in D,\quad f(z)<f(x)+ d_x(x,z).
\end{align*}\end{theo}\begin{proof}
Considering the order $\le_f$ given by   \eqref{def1.ord-Fqg} (with $f$ instead of $\varphi$), it follows that
$$
S_f(x)=S_-(x)=\{y\in X: y\le_f x\}\,.$$

Let $x_0\in \dom f$. The inequality

$$f(y)+d(y,x_0)\le f(x_0)\,,$$
shows that $S_f(x_0)\subseteq \dom f,$ so that  $f(S_f(x_0))\subseteq\mathbb{R}.$
Since $S_f(x_0)$ is  sequentially  $\mathcal U$-closed it follows that it is  sequentially right $\mathcal U$-$K$-complete too, so we can apply
 Proposition \ref{c.BB-Fqg} to conclude that  the set $  S_f(x_0)$ contains  a $\,\le_f$-minimal element $z$.

 The relation $z\in S_f(x_0)$ is equivalent to (i).

    To prove  (ii) suppose that $x\in X\smallsetminus\{z\}.$

  If $x\in S_f(x_0)\smallsetminus\{z\},$ then, by the minimality of $z$, the inequality $x\le_f z$ fails, so there exists $d_x\in D$ such that
  $$
 f(x)+ d_x(x,z)>f(z)\,.$$

  Suppose now that $x\in X\smallsetminus S_f(x_0)$ and that (ii) fails for this $x$, that is,
  $$
  \forall \delta\in D,\;\; f(z)\ge f(x)+\delta(x,z)\,.$$

  Given $d\in D$ choose $d'\in D$ according to (QF3). Then
  \begin{align*}
    d(x,x_0)&\le d'(x,z)+d'(z,x_0)\\
    &\le f(z)-f(x)+f(x_0)-f(z)\\
    &=f(x_0)-f(x)\,,
  \end{align*}
  implying $x\in S_f(x_0)$, in contradiction to the hypothesis.
  \end{proof}

  Taking into account  Remark \ref{re.lsc-clos}, one obtains the following result.
  \begin{corol}
    The conclusions (i) and (ii) of Theorem \ref{t1.EkVP-Fqg} hold if condition (a) is replaced by
    \begin{itemize}\item[\rm(a$'$)] the function $f$ is sequentially $\mathcal U$-lsc.\end{itemize}
  \end{corol}

  The Ekeland Variational Principle from Theorem \ref{t1.EkVP-Fqg} implies the quasi-uniform analog of an  apparently stronger version proved by Fang \cite{fang96} (see also \cite{hamel01}).
  \begin{corol}\label{c1.EkVP-Fqg}  Suppose, in addition to the hypotheses of Theorem \ref{t1.EkVP-Fqg}, that an increasing function $\xi:D\to(0,\infty)$ is given.
  For  $\varepsilon>0$ suppose that   $x_0\in X$  satisfies the condition $f(x_0)\le\inf f(X)+\varepsilon$. Then  there exists $z\in X$ such that
\begin{align*}
  &{\rm (j)}\;\; \forall d\in D,\quad f(z)+ \varepsilon \xi(d)d(z,x_0)\le f(x_0);\\
   &{\rm (jj)}\;\; \forall d\in D,\quad d(z,x_0)\le\frac1{\xi(d)};\\
   &{\rm (jjj)}\;\; \forall x\in X\setminus \{z\},\; \exists  d_x\in D,\quad f(z)<f(x)+ \varepsilon \xi(d_x)d_x(x,z).
\end{align*}
  \end{corol}\begin{proof} For $d\in D$ put $\tilde d=\varepsilon \xi(d)d$ and $\widetilde D=\{\tilde d: d\in D\}. $ It is easy to check that  $\widetilde D$ satisfies the conditions (QF1)--(QF3) from   Definition \ref{def.Fqg}.

  Furthermore, $\mathcal U_{\widetilde D}=\mathcal U_D=\mathcal U$. This follows from the equalities
  $$
  V_{\tilde d,\gamma}=V_{d,\frac \gamma{\varepsilon \xi(d)}}\quad\mbox{and}\quad V_{d,\gamma'}=V_{\tilde d,\varepsilon \xi(d) \gamma'}\,,$$
  valid for all $d\in D$ and $\gamma,\gamma'>0$.

  Apply Theorem \ref{t1.EkVP-Fqg} to the quasi-uniform space $(X,\mathcal U)$  with the quasi-uniformity generated by the $F$-quasi-gauge $\widetilde D$ and the function $f$.
  Then (i) is equivalent to (j) and (ii) to (jjj).

  Let $\alpha=\inf f(X)$. Then by (j), for all $d\in D$,
    $$
  \varepsilon \xi(d)d(z,x_0)\le f(x_0)-f(z)\le \alpha+\varepsilon-\alpha=\varepsilon\,,$$
  implying $\xi(d)d(z,x_0)\le1\,$ which is equivalent to (jj).
  \end{proof}

  Caristi-Kirk fixed point theorem takes the following form in this case.

  \begin{theo}\label{t.Car-Fqg}  Let $(X,\mathcal U)$ be a quasi-uniform space whose quasi-uniformity is generated by an $F$-quasi-gauge $D$, (i.e. $\,\mathcal U=\mathcal U_D$).
 Suppose  that the function  $\varphi:X\to \mathbb{R}\cup\{\infty\}$ is proper, bounded below and  that the set $S_\varphi(x)$, given by \eqref{def.S(x)}, is sequentially $\mathcal U$-closed for every $x\in X$.
Suppose further that  $(X,\mathcal U)$ is   sequentially right $K$-complete, the topology $\tau(\mathcal U)$ is $T_1$  and that $F:X\rightrightarrows X  $ is  a set-valued mapping with nonempty values.
\begin{enumerate}\item[\rm 1.]
If
\begin{equation}\label{eq1.Car-Fqg}
 \forall x\in X,\; \exists y\in F(x)\;\mbox{ s.t. }\; \forall d\in D,\;\varphi(y)+d(y,x)\le \varphi(x)\,,\end{equation}
then $F$ has a fixed point in $X$, i.e. there exists $z\in X$ such that $z\in F(z).$
\item[\rm 2.]
If
\begin{equation}\label{eq2.Car-Fqg}
\;\forall x\in X, \;\forall y\in F(x), \; \forall d\in D,\; \varphi(y)+d(y,x)\le \varphi(x)\,,\end{equation}
then $F$ has a stationary point in $X$, i.e. there exists $z\in X$ such that $F(z)=\{z\}.$
     \end{enumerate}\end{theo} \begin{proof}
     Notice that the condition \eqref{eq1.Car-Fqg} is equivalent to
     \begin{equation}\label{eq3.Car-Fqg}
 \forall x\in X,\;\; S_\varphi(x)\cap F(x)\ne\emptyset\,.
     \end{equation}
     \
      By Theorem \ref{t1.EkVP-Fqg} applied to $X$ and $\varphi$, there exists $z\in X$ such that,   for all $y\in X$,
   $$
   y\ne z\;\Rightarrow\; \exists d_y\in D,\;\; \varphi(z)<\varphi(y)+d_y(y,z)\,.$$

     It follows that no point $y\in F(z)\sms\{z\}$ belongs to $S_\varphi(z)$, so that,  by \eqref{eq3.Car-Fqg}, $S_\varphi(z)\cap F(z) =\{z\} $. Hence, $z\in F(z).$

     In the second case, the condition \eqref{eq2.Car-Fqg}  implies $F(z)\subseteq S_\varphi(z)\,,$  so that
     $$
     F(z)=S_\varphi(z)\cap F(z) =\{z\}\,.$$
     \end{proof}

     In the case of single-valued mappings the Caristi-Kirk fixed point theorem takes the following form.
     \begin{corol}\label{c.Car-Fqg}
     Suppose that $(X,\mathcal U),\, D$ and $\varphi:X\to\mathbb{R}\cup\{\infty\}$ satisfy the hypotheses of Theorem \ref{t.Car-Fqg}.
     If $f:X\to X$ is a mapping such that
     \begin{equation*}
\forall d\in D,\quad   \varphi(f(x))+d(f(x),x)\le \varphi(x)\,,\end{equation*}
 for all $ x\in X$, then $f$ has a fixed point in $X$.
     \end{corol}

     Again, taking into account Remark \ref{re.lsc-clos}, one obtains the following result.
     \begin{corol}\label{c2.Car-Fqg}
     The conclusions of Theorem \ref{t.Car-Fqg} and Corollary \ref{c.Car-Fqg} remain true if the closedness hypothesis on the set $S_\varphi(x)$ is replaced by
     the condition
     \begin{itemize}\item the function $\varphi$ is sequentially $\mathcal U$-lsc.\end{itemize}
     \end{corol}

    \begin{remark}  \begin{enumerate}\item[\rm 1.] In the uniform case the analog of Corollary \ref{c2.Car-Fqg} was obtained by Hamel   \cite{hamel05} and in  $F$-gauge spaces by Fang \cite{fang96}     (see also \cite{hamel01}).
    \item[\rm 2.] As in the case of Corollary \ref{c1.EkVP-Fqg}, the inequalities in \eqref{eq1.Car-Fqg} and  \eqref{eq2.Car-Fqg} can be replaced with
    $$\varphi(y)+\xi(d)d(y,x)\le \varphi(x)\,,$$
    where $\xi:D\to(0,\infty)$ is an increasing function,   obtaining the same conclusions.
    \end{enumerate}\end{remark}

    Takahashi's principle (see \cite{taka91}  or \cite{Taka00}) has an analog in this case too.
     \begin{theo}\label{t.Taka-Fqg}  Let $(X,\mathcal U)$ be a quasi-uniform space whose quasi-uniformity is generated by an $F$-quasi-gauge $D$, (i.e. $\mathcal U=\mathcal U_D$)
 and $f:X\to \mathbb{R}\cup\{\infty\}$.
Suppose that  $(X,\mathcal U)$ is   sequentially right $K$-complete, the topology $\tau(\mathcal U)$ is $T_1$  and $f  $ is  proper, bounded below  and  that the set $S_f(x)$ is sequentially $\mathcal U$-closed for every $x\in X$.  If for every  $x\in X$   with $f(x)>\inf f(X)$ there exists $y\in X\smallsetminus\{x\}$ and that
\begin{equation}\label{eq1.Taka-Fqg}
\forall d\in D,\quad f(y)+d(y,x)\le f(x)\,,\end{equation}
then there exists $ z\in X$ with $f(z)=\inf f(X)$.
        \end{theo}  \begin{proof}  We show that the result can be obtained as a consequence of the Caristi-Kirk fixed point theorem (Theorem \ref{t.Car-Fqg}).
          Let $F:X\rightrightarrows X$ be defined by
          $$
          F(x)=S_f(x)=\{y\in X: f(y)+d(y,x)\le f(x)\;\mbox{for all}\; d\in D\},\quad x\in X.$$

          Then $F(x)\ne\emptyset$, as $x\in F(x)$. Since $F(x)\cap S_f(x)=S_f(x)$ it follows that  \eqref{eq2.Car-Fqg} is satisfied (with $f$ instead of $\varphi$). By Theorem \ref{t.Car-Fqg} there exists
          $z\in X$ such that $F(z)=\{z\}$, i.e. $S_f(z)=\{z\}$. If $f(z)>\inf f(X)$, then, by \eqref{eq1.Taka-Fqg}, $S_f(z)\sms\{z\}\ne \emptyset,$ a contradiction. Consequently, $f(z)=\inf f(X)$.
        \end{proof}

        A   result similar   to  Takahashi's  minimization principle  (Theorem \ref{t.Taka-Fqg}) in metric spaces, under a slightly relaxed condition on the function $f$, was found by Arutyunov and  Gel'man \cite{arut-gel09} and Arutyunov \cite{arut15}.  For some   further recent developments see \cite{arut-zhuk19a} and \cite{arut-zhuk19b}.

        We    present first the case   when the result can be obtained from Ekeland Variational Principle.  The general case will be treated later in Theorem \ref{t2.Arut}.

 \begin{theo}[\cite{arut15}, Theorem 3]\label{t.Arut-Fqg} Let $(X,\mathcal U)$ be a quasi-uniform space whose quasi-uniformity is generated by an $F$-quasi-gauge $D$, (i.e. $\mathcal U=\mathcal U_D$)
 and $f:X\to [\alpha,\infty]$.
Suppose that  $(X,\mathcal U)$ is   sequentially right $K$-complete, the topology $\tau(\mathcal U)$ is $T_1$  and $f  $ is  proper, bounded below,  with $\alpha:=\inf f(X)>-\infty$, and that the set $S_f(x)$ is sequentially $\mathcal U$-closed for every $x\in X$.  Suppose also that there exists $\gamma>0$ such that for every $x\in X$ with $f(x)>\alpha$ there is $x'\in X\smallsetminus\{x\}$  satisfying
\begin{equation}\label{eq1.Arut}
 \forall d\in D,\quad f(x')+\gamma d(x',x)\le f(x)\,.
\end{equation}
Then, for every $x_0\in \dom f$   there exists $\bar x\in X$  such that
 \begin{equation}\label{eq2.Arut}\begin{aligned}
   {\rm (a) }&\quad  f(\bar x)=\inf f(X)=\alpha\quad\mbox{ and }\\
   \quad {\rm (b) }& \quad d(\bar x,x_0)\le \frac{f(x_0)-\alpha}{\gamma}\,.
 \end{aligned}
\end{equation}\end{theo}\begin{proof}
 If $f(x_0)=\alpha$, then $f(x_0)=\inf f(X)$, and   there is nothing to prove.

Suppose  $f(x_0)>\alpha$ and apply  Theorem \ref{t1.EkVP-Fqg} with $d$ replaced by $\frac{\varepsilon}{\lambda} d, $ where  $\varepsilon:=f(x_0)-\alpha>0$    and $\lambda=\varepsilon/\gamma,$  for all $d\in D$. Then there exists $\bar x\in X$ such that
\begin{equation}\label{eq9.Arut}\begin{aligned}
{\rm (a)}& \;\quad \;\forall x\in X\smallsetminus\{\bar x\},\; \exists d_x, \; f(\bar x) < f(x)+\frac{\varepsilon}{\lambda}d_x(x,\bar x);\\
{\rm (b)}&\; \quad \forall d\in D,\;  f(\bar x)+\frac{\varepsilon}{\lambda}d(\bar x,x_0)\le f(x_0)\,.
\end{aligned}\end{equation}

If $f(\bar x)>\alpha$, then, by \eqref{eq1.Arut}, there exists $x\ne\bar x$ such that
$$
 f(x)+\gamma d_{x}(x,\bar x)\le f(\bar x)\,,$$
 in contradiction to \eqref{eq9.Arut}.(a) (because $\gamma=\varepsilon/\lambda$).

 Hence $f(\bar x)=\alpha=\inf f(X)$.

 Now, by \eqref{eq9.Arut}.(b),
 $$
\forall d\in D,\;\; \gamma d(\bar x,x_0)\le f(x_0)-\alpha\,,$$
 which yields \eqref{eq2.Arut}.(b).
\end{proof}

 \begin{remark}
  Replacing the metric $d$ with the equivalent one $\tilde d=\gamma d,$ the condition \eqref{eq1.Arut} becomes Takahashi's condition \eqref{eq1.Taka-Fqg}.
  Arutyunov \cite{arut15} calls \eqref{eq1.Arut}  a \emph{ Caristi-type condition}.
\end{remark}

Finally, we prove an equilibrium version of Ekeland Variational Principle. Let $(X,\mathcal U)$ be a quasi-uniform space with the quasi-uniformity $\mathcal U$ generated by an F-quasi-gauge  $D$.

Suppose that the mapping  $F:X\times X\to\mathbb{R}\cup\{\infty\}$ satisfies the following conditions:
\begin{equation}\label{def.eq-fcs}\begin{aligned}
  {\rm (E1)}&\quad \forall x\in X,\; F(x,x)=0;\\
  {\rm (E2)}&\quad \forall x,y,z\in X,\;F(x,z)\le F(x,y)+F(y,z);\\
  {\rm (E3)}&\quad \forall x\in X,\; \mbox{the function}\;  F(x,\cdot) \;\mbox{is sequentially } \mathcal U\mbox{-lsc};\\
  {\rm (E4)}&\quad \exists x_0\in X,\; \inf\{F(x_0,y):y\in X\}>-\infty\,.
\end{aligned}\end{equation}

For $x\in X$ let
\begin{equation*}
S(x)=\{y\in X:F(x,y)+d(y,x)\le 0\;\mbox{for all}\; d\in D\}\,.\end{equation*}

Following the paper \cite{ans12} let us consider the weaker condition:
\begin{equation*}
  {\rm (E3a)}\quad \mbox{the level set } S(x) \mbox{ is closed for every } x\in X.\smallskip
\end{equation*}

It is easy to check that \begin{equation}\label{eq.E3}{\rm (E3)}\;\Rightarrow\; {\rm(E3a)}.
\end{equation}
Indeed, if $(y_n)$ is a sequence  in $S(x)$ that is $\mathcal U$-convergent to some $y\in X$, then, for every $\delta\in D$,
$$
F(x,y_n)+\delta(y_n,x)\le 0\,,$$
for all $n\in\mathbb{N}.$  Given  $d\in D$ and let  $d'\in D$ be given by condition (QF3) from Definition \ref{def.Fqg}. Then, by the $\mathcal U$-lsc of the functions $F(x,\cdot)$ and Proposition \ref{p2.top-Fqg}, item 2,
$$ F(x,y)+d(y,x)\le \liminf_n[F(x,y_n)+d'(y_n,x)]\le 0\,,$$
i.e. $y\in S(x)$. This  shows that  $S(x)$ is sequentially $\mathcal U$-closed.

The following version of the  Oettli-Th\'era theorem extends to F-quasi-gauge  spaces a result  obtained in \cite{ans12} in the case of metric spaces.

\begin{theo}\label{t.OT-Fqg}
Suppose that  $(X,\mathcal U)$ is a sequentially right $K$-complete quasi-uniform space whose quasi-uniformity is generated by an $F$-quasi-gauge $D$ (i.e. $\mathcal U=\mathcal U_D$) and that the topology $\tau(\mathcal U)$ is $T_1$.
Let the  function   $F:X\times X\to \mathbb{R}\cup\{\infty\}$ and $x_0\in X$ satisfy the conditions (E1),(E2), (E3a) and (E4) from above. Then there exists $z\in S(x_0)$ such that
\begin{equation}\label{eq1.OT-Fqg}
\forall x\in X\smallsetminus\{z\},\; \exists d_x\in D\; \mbox{ s.t. } \; F(z,x)+d_x(x,z)>0\,.\end{equation}
\end{theo}\begin{proof}
  We  apply Theorem \ref{t1.EkVP-Fqg} to the function $f:X\to\mathbb{R}\cup\{\infty\},\, f(x)=F(x_0,x),\, x\in X,$ to obtain an element    $z\in X$ satisfying the conditions (i) and (ii) from that theorem.

  By (i) and (E1), $\, z\in S(x_0)$.

  By (ii), for every $x\in X\smallsetminus\{z\}$ there exists $d_x\in D$ such that
  $$
  F(x_0,z)<F(x_0,x)+d_x(x,z)\,.$$

  But then,  by (E2),
  \begin{align*}
  F(x_0,z)&<F(x_0,x)+d_x(x,z)\\
  &\le F(x_0,z)+F(z,x)+d_x(x,z)\,,
  \end{align*}
  yielding \eqref{eq1.OT-Fqg}.
\end{proof}

Taking into account the implication from \eqref{eq.E3}, one obtains the following consequence of Theorem \ref{t.OT-Fqg}.
\begin{corol}
The conclusion of Theorem \ref{t.OT-Fqg} remains valid if one supposes that the conditions (E1)--(E4) are satisfied.
  \end{corol}

\begin{remark}This  result and its equivalence to EkVP was proved by Oettli and Th\'era \cite{oetli-thera93} in complete metric spaces, by Hamel  \cite{hamel05} in uniform spaces,
by Fang \cite{fang96} and Hamel \cite{hamel01} in $F$-gauge spaces
and by Al-Homidan, Ansari and Kassay \cite{kassay19} in quasi-metric spaces.\end{remark}

\begin{remark}\label{re.Ek-OT} The Ekeland Variational principle   (Theorem \ref{t.OT-Fqg}) can be obtained from Theorem \ref{t.OT-Fqg} by putting  $F(x,y)=f(y)-f(x)$   (see Theorem \ref{t.Ek-OT} for details).
\end{remark}

\begin{remark}
Theorems \ref{t.Car-Fqg},   \ref{t.Taka-Fqg},  \ref{t.Arut-Fqg}  and \ref{t.OT-Fqg} also admit extensions in the spirit of Corollary \ref{c1.EkVP-Fqg}.
The exact statements are left  as an easy exercise.\end{remark}

   \section{The equivalence of principles}\label{S.equiv}
We prove   the equivalence between Ekeland, Takahashi and Caristi principles.
\begin{theo} \label{t1.wEk-Taka-Car}
Let  $D$ be an $F$-quasi-gauge on a set $X$ generating a $T_1$ quasi-uniformity $\mathcal U=\mathcal U_D$ and
  $\varphi:X\to\mathbb{R}\cup\{\infty\}$  a proper bounded below   function.
 Then the following statements are equivalent.
 \begin{enumerate}\item[\rm 1.] {\rm (Ekeland)}\;There exists $z\in X$ such that
 \begin{equation}\label{eq1.Ek3-Fqg}
\forall x\in X\setminus \{z\},\; \exists  d_x\in D,\quad \varphi(z)<\varphi(x)+ d_x(x,z)\,.
\end{equation}
  \item[\rm 2.] {\rm (Caristi-Kirk)}\; Any mapping $f:X\to X$ satisfying the condition
     \begin{equation}\label{eq1.Car3-Fqg}
\forall d\in D,\; \forall x\in X,\quad  \varphi(f(x))+d(f(x),x)\le \varphi(x)\,,
\end{equation}
  has a fixed point in $X$.
 \item[\rm 3.]  {\rm (Takahashi)}\;If for every  $x\in X$   with $\varphi(x)>\inf \varphi(X)$ there exists $y\in X\smallsetminus\{x\}$ such that
\begin{equation}\label{eq1.Taka3-Fqg}
\forall d\in D,\quad \varphi(y)+ d(y,x)\le \varphi(x)\,,\end{equation}
then there exists $ z\in X$ with $\varphi(z)=\inf \varphi(X)$.
\end{enumerate}\end{theo}\begin{proof}
  1\;$\Rightarrow\;$2.\;

  Let $z\in X$ be such that   \eqref{eq1.Ek3-Fqg} holds. Suppose that  $f:X\to X$ satisfies \eqref{eq1.Car3-Fqg}. Then
  \begin{equation}\label{eq1b.Car3-Fqg}
  \forall d\in D,\;\; \varphi(f(z))+d(f(z),z)\le \varphi(z)\,.\end{equation}
If $f(z)\ne z$, then there exists $d_z\in D$ such that
$$
\varphi(f(z))+d_z(f(z),z)>\varphi(z)\,,
$$
in contradiction to \eqref{eq1b.Car3-Fqg}. Consequently,
 we must have  $f(z)=z$.\smallskip

  $\neg$1\;$\Rightarrow\,\neg 2$.\;

  If \eqref{eq1.Ek3-Fqg} does not hold, then
  $$
  \forall z\in X,\;\exists y_z\in X\smallsetminus\{z\}\; \mbox{ s.t. }\; \forall d\in D,\quad \varphi(y_z)+ d(y_z,z)\le \varphi(z)\,.$$

  Defining $f:X\to X$ by $f(z)=y_z,\, z\in X$, it follows that $f$ satisfies \eqref{eq1.Car3-Fqg} but has no fixed point. \smallskip

  1\;$\Rightarrow\;$3.\;

  Let $z\in X$ be such that   \eqref{eq1.Ek3-Fqg} holds. If $\varphi(z)>\inf\varphi(X)$, then, by \eqref{eq1.Taka3-Fqg}, there exists
$  y\ne z$  such that
$$
\forall d\in D,\; \; \varphi(y)+ d(y,z)\le \varphi(z)\,,$$
in contradiction to \eqref{eq1.Ek3-Fqg}.  It follows $\varphi(z)=\inf \varphi(X)$.\smallskip

 $\neg$1\;$\Rightarrow \neg 3$.\;

 Supposing that  \eqref{eq1.Ek3-Fqg} fails, then
 $$
    \forall z\in X,\; \exists y_z\ne z,\; \mbox{ s.t. }\; \forall d\in D,\; \varphi(y_z)+d(y_z,z)\le \varphi(z).$$

    This shows that $\varphi$ satisfies \eqref{eq1.Taka3-Fqg} and that $z\le_\varphi y_z$ with respect to the order given by \eqref{def1.ord-Fqg}.

     Since the topology $\tau(\mathcal U)$ is $T_1$,    Proposition \ref{p1.ord-Fqg}   implies that the function $\varphi$ is strictly decreasing, so that  $\varphi(y_z)<\varphi(z).$

    Consequently, for every
    $z\in X$ there exists $y_z\in X$ with $\varphi(y_z)<\varphi(z)$, implying that there is no $z\in X$  with $\varphi(z)=\inf\varphi(X)$, i.e. 3 fails.
\end{proof}
\begin{remark}
The version of Ekeland Variational Principle formulated in the first statement of Theorem \ref{t1.wEk-Taka-Car} is called sometimes the weak form of
Ekeland Variational Principle. Oettli-Th\'era result is equivalent to the full version of Ekeland Variational Principle (Theorem \ref{t1.EkVP-Fqg}).
I don't know if Arutyunov's result implies Ekeland Variational Principle (Theorem \ref{t1.EkVP-Fqg}).
\end{remark}

As was noticed in Remark \ref{re.Ek-OT},  Ekeland and Oettli-Th\'era principles are actually equivalent.
\begin{theo}\label{t.Ek-OT}
Theorems \ref{t1.EkVP-Fqg}  and \ref{t.OT-Fqg} are equivalent.
\end{theo}\begin{proof}

  Theorem \ref{t1.EkVP-Fqg} \;$\Rightarrow$\;   Theorem \ref{t.OT-Fqg}.

  The proof given to Theorem  \ref{t.OT-Fqg} shows the validity of this implication.\smallskip

  Theorem \ref{t.OT-Fqg} \;$\Rightarrow$\;   Theorem  \ref{t1.EkVP-Fqg} .

Suppose that the hypotheses of Theorem \ref{t1.EkVP-Fqg} are fulfilled. Taking  $F(x,y)=f(y) -f(x)$ for $x,y\in X$,  it follows that
\begin{equation}\label{eq.Ek-OT}
S(x)=S_f(x)\,,
\end{equation}
for all $x\in X$. Indeed,
\begin{align*}
S(x)&=\{y\in X: F(x,y)+d(x,y)\le 0,\; \forall d\in D\}\\
&=\{y\in X:f(y)+d(x,y)\le f(x),\; \forall d\in D\}=S_f(x)\,.
\end{align*}

It is easy to check that the conditions (E1), (E2), (E4) from \eqref{def.eq-fcs} and (E3a) are satisfied by $F$.  Then there exists $z\in S(x_0)$ satisfying  \eqref{eq1.OT-Fqg}. By \eqref{eq.Ek-OT}, $z\in S(x_0)$
is equivalent to (i) and   \eqref{eq1.OT-Fqg}   is equivalent to (ii) from Theorem   \ref{t1.EkVP-Fqg}.
\end{proof}

\section{Some classes of functions}\label{S.fcs}

In order to obtain stronger maximality or minimality principles some authors considered more general classes of functions than the lsc ones.
 A discussion on  three of these  versions  is done in the paper \cite{bcs18}.  We shall present now these conditions.

\begin{defi}
 Let $(X,d)$ be a quasi-metric space and $\varphi: X \to \mathbb{R}\cup\{+\infty\}$ be an extended real-valued function on $X$.
 \begin{itemize}\item
 The function $\varphi$ is said to be {\it strict-decreasingly lsc}  (see \cite{bcs18}) if for every $d$-convergent sequence $(x_n)$ such that the sequence $(\varphi(x_n))$ is strictly decreasing, one has
\begin{equation}\label{def.lsc-mon}
\varphi(y)\le  \lim_{n\to\infty} \varphi(x_n) , \; \forall\; y \in \overrightarrow{\{x_n\}},
\end{equation}
where $\overrightarrow{\{x_n\}} = \{y \in X:  \lim_{n\to\infty} d(y,x_n) = 0\}$ is the collection of $d$-limits of the sequence $\{x_n\}$.
\item
A related notion is that of  {\it decreasingly  lsc function} considered by  Kirk and Saliga  \cite{kirk-sal01} (called by them  \emph{lower semicontinuity from above})   meaning that \eqref{def.lsc-mon} holds for every  sequence $(x_n)$  $d$-convergent to $y$ and such that $\varphi(x_{n+1})\le\varphi(x_n),\forall n\in\mathbb{N}.$

\item  Following  \cite{karap-romag15} we call the function  $\varphi$  {\it nearly lsc} if whenever a sequence $(x_n)$ of pairwise distinct points in $X$ is $d$-convergent to   $x$, then $\varphi(x) \leq \liminf_{n\to\infty} \varphi(x_n)$.\end{itemize}

 The corresponding concepts for the conjugate quasi-metric $\bar d$  can be defined analogously.
\end{defi}

 A characterization on nearly lsc functions on quasi-pseudometric spaces in terms of the monotony with respect to the specialization order is given in \cite{cobz19} -- a function is lsc if and only if  it is nearly lsc and monotone with respect to the specialization order.

 The specialization $\le_\tau$ order in a topological space $(T,\tau)$ is given by
 $$
 s\le_\tau t\iff s\in \overline{\{t\}}\,,$$
 where $\overline S$ denotes the closure of a subset $S$ of $T$.

 In general, it is  a preorder, i.e.   a  reflexive and transitive relation. It is an order (also antireflexive) if and only if the topology $\tau$ is $T_0$.
    If the topology $\tau$ is $T_1$, then$ \overline{\{t\}}=\{t\}$ for all $t\in T$, so that  the specialization order becomes equality and   the nearly lower semicontinuity goes back to  lower semicontinuity. For other topological properties related to the specialization order in $T_0$ topological spaces,  see \cite[Section 4.2]{Larrecq}.

It is worth mentioning that the class of strict-decreasingly lsc functions is broader than the union of that of the  decreasingly lsc functions and that of the   nearly lower semicontinuous functions.
 Since the strict-$\varphi$-decreasing requirement of the sequence $(x_n)$ implies that $x_n\neq x_m$ for all $n\neq m$, every nearly lower semicontinuous function is strict-$\varphi$-decreasingly lsc.


Let us provide some examples of  functions satisfying these conditions.

\begin{example} {\bf(decreasingly  lsc, strict-decreasingly  lsc  and lsc functions).}

 {\rm(a)} Consider the  functions $\varphi,\,\varphi_1 : (\mathbb{R},|\cdot|) \rightarrow (\mathbb{R},|\cdot|)$ given by
 \begin{eqnarray*}
 \varphi(x) := \begin{cases}
 x & \mbox{\rm if }\; x \geq 0\\
 -1 & \mbox{\rm if }\; x < 0\\
 \end{cases}\quad\mbox{and}\quad
  \varphi_1(x) := \begin{cases}
 -x & \mbox{\rm if }\; x > 0\\
 1 & \mbox{\rm if }\; x \le 0\,.\\
 \end{cases}
 \end{eqnarray*}
 Then $\varphi$ is strict-decreasingly lsc at $0$ , but not decreasingly lsc (and so not lsc) at 0,  since    the  sequence   $x_n = -1/n,\,n\in\mathbb{N},$ is convergent to 0 and  $\,\varphi(x_n) = -1$ for all $n\in \mathbb{N}$, so that $\lim_{n\to\infty} \varphi(x_n) = -1 < 0 = \varphi(0)$.

 The function $\varphi_1$ is decreasingly lsc at 0, but not lsc. Indeed,  there are no sequences $x_n\to 0$ with $\{\varphi_1(x_n)\}$ strictly decreasing. If $x_n\to 0$ and $\varphi_1(x_{n+1})\le\varphi_1(x_n)$ for all $n$, then $\varphi_1(x_n)=1$ for sufficiently large $n$, so that $\lim_{n\to{\infty}}\varphi_1(x_n)=1=\varphi(0)$. The function $\varphi_1$ is not lsc at 0 because $\lim_{x\searrow 0}\varphi_1(x)=0<1=\varphi_1(0)$.
 \vspace*{.05in}

{\rm (b)} The   function $\varphi(x) = x$ is not strict-decreasingly $q_4$-lower semicontinuous in the quasi-metric space $([0,1],q_4)$,  where $q_4$ is defined by
$$
q_4(x,y) = \begin{cases}
y - x &\mbox{ if } \; x\leq y;\\
1 + y-x &\mbox{ if } \; x > y \; \mbox{ but }\; (x,y) \neq (1,0);\\
1 &\mbox{ if } \; (x,y) = (1,0).
\end{cases}
$$

In this space, the strict-$\varphi$-decreasingly   sequence $(x_n)$ with $x_n = \frac{1}{n}$ has two  limits $x_\ast = 0$ and $y_\ast = 1$, because
$q_4(0,\frac1n)=q_4(1,\frac1n)=\frac1n\to 0.$  Since
$$
\varphi(1)=1>0=\lim_{n\to\infty} \varphi(x_n) \,,
$$
the function  $\varphi$ is not strict-decreasingly lower semicontinuous in $(X,q_4)$. Incidentally, this furnishes an example of a  $T_1$-quasi-metric space  where the uniqueness condition for  limits does not hold. It is easy to check that the sequence $x_n=1/n$ is also right $K$-Cauchy.

The quasi-metric $q_4$ was introduced in  \cite[Example~3.16]{lwa11}.

\vspace*{.05in}
 {\rm (c)}   The everywhere discontinuous function $\varphi(x) = 0$ for $x \in \mathbb{Q}$ and $\varphi(x) = 1$ for $x \in \mathbb{R}\setminus\mathbb{Q}$,
defined on $(\mathbb{R},|\cdot|)$, is strict-decreasingly lower semicontinuous because there are no strictly $\varphi$-decreasing sequences. For every  $x\in\mathbb{R}\setminus\mathbb{Q}$ the function $\varphi$ is not lower semicontinuous at $x$, because
$\varphi(x) = 1 > 0 = \liminf_{u\to{x}} \varphi(u)$.   Also, for every $x \in \mathbb{Q}$ it is not upper semicontinuous at $x$.
\end{example}

Arutyunov and Gel'man  \cite{arut-gel09} and Arutyunov \cite{arut15} introduced another class of functions   on a metric space $(X,d)$, namely functions satisfying condition \eqref{eq1a.Arut2} from below (with convergence with respect to $d$ instead of the $\tau(\mathcal U)$-convergence) and     proved the existence of minima for functions in this class defined on complete metric spaces. For further results, see \cite{arut-zhuk19a} and \cite{arut-zhuk19b}.

We show that these results can be extended to quasi-uniform spaces.

\begin{theo}[Arutyunov \cite{arut15}]\label{t2.Arut} Let $P$ be a family of quasi-pseudometrics on a set $X$ generating a quasi-uniformity $\mathcal U=\mathcal U_P$ such that   $(X,\mathcal U)$ is sequentially right $K$-complete.  Let $\eta:\mathbb{R}_+\to \mathbb{R}_+$ be a usc function such that
\begin{equation}\label{eq1.eta}
\eta(t)<t\quad\mbox{for all}\quad t>0\,.\end{equation}

Suppose that $f:X\to[0,\infty]$ is a function satisfying the conditions:
\begin{itemize}\item[\rm (a)] for every sequence $(x_n)$ in $X$ and every  $x\in X$
\begin{equation}\label{eq1a.Arut2}
x_n\xrightarrow{\tau(\mathcal U)} x\;\mbox{ and }\; f(x_n)\to 0\;\Rightarrow f(x)=0;\end{equation}
\item[\rm (b)] there exists $\gamma>0$ such that for every $x\in X$ with $f(x)>0$ there exists $x'\in X$ such that
\begin{equation}\label{eq2.Arut2}\begin{aligned}
  {\rm (i)}\;\; &\forall d\in P,\;\; f(x')+\gamma d(x',x)\le f(x),\\
   {\rm (ii)}\;\;&f(x')\le \eta(f(x)).
\end{aligned}\end{equation}
\end{itemize}

Then, for every $x_0\in \dom f$, there exists $\bar x\in X$ such that
\begin{equation}\label{eq3.Arut2}\begin{aligned}
  {\rm (i)}\;\; &f(\bar x)=0=\inf f(X);\\
  {\rm (ii)}\;\; &\forall d\in P,\;\; d(\bar x,x_0)\le\frac{f(x_0)}\gamma\,.
\end{aligned}\end{equation}\end{theo}
\begin{proof}
Observe first that, by \eqref{eq1.eta},
\begin{equation}\label{eq2.eta}
\alpha\le \eta(\alpha)\;\Rightarrow\; \alpha =0\,,
\end{equation}
for each $\alpha\ge 0.$

Let $x_0\in \dom f$. If $f(x_0)=0$, then we are done (with $\bar x=x_0$).

Suppose that $f(x_0)>0$. Then by the condition (b) of the theorem, there exists $x_1\in X$ such that
\begin{equation*} \begin{aligned}
  {\rm (i)}\;\; &\forall d\in P,\;\;  \gamma d(x_1,x_0)\le f(x_0)-f(x_1)),\\
   {\rm (ii)}\;\;&f(x_1)\le \eta(f(x_0)).
  \end{aligned} \end{equation*}

   By (ii) $x_1\in \dom f$. If $f(x_1)=0$ then take $\bar x=x_1$ and stop.

   Suppose that we have found $x_0,x_1,\dots,x_n$ in $\dom f$ such that
   \begin{equation}\label{eq4.Arut2} \begin{aligned}
  {\rm (i)}\;\; &\forall d\in P,\;\;  \gamma d(x_{k+1},x_k)\le f(x_k)-f(x_{k+1}),\\
   {\rm (ii)}\;\;&f(x_{k+1})\le \eta(f(x_k)),
  \end{aligned} \end{equation}
  for $k=0,1,\dots,n-1$ and $f(x_k)>0$ for $k=0,1,\dots,n.$

 Then choose $x_{n+1}\in X$ such that
 \begin{equation*}\begin{aligned}
  {\rm (i)}\;\; &\forall d\in P,\;\;  \gamma d(x_{n+1},x_n)\le f(x_n)-f(x_{n+1})),\\
   {\rm (ii)}\;\;&f(x_{n+1})\le \eta(f(x_n)).
   \end{aligned} \end{equation*}

   If  $f(x_{n+1})=0$, then take $\bar x=x_{n+1}$. Then $f(\bar x)=0=\inf f(X)$ and, by \eqref{eq4.Arut2}.(i),
   \begin{align*}
     \gamma d(\bar x,x_0)&\le \gamma[d(x_{n+1},x_n)+\dots+d(x_1,x_0)]\\
     &\le f(x_n)-f(x_{n+1})+\dots +f(x_0)-f(x_1)\\
     &=f(x_0)-f(x_{n+1})=f(x_0)\,,
   \end{align*}
   for all $ d\in P$, implying \eqref{eq3.Arut2}.(ii).

    If $f(x_{n+1})>0$ at any step,  then we find a sequence $(x_k)_{k\in\mathbb{N}_0}$ in $\dom f,\,$  satisfying the conditions \eqref{eq4.Arut2} for all $k\in\mathbb{N}_0$.

    Observe that, by \eqref{eq4.Arut2}.(i), the sequence $(f(x_k))$ is decreasing, so that there exists $\lim_{k\to\infty}f(x_k)=\alpha\ge 0.$
    By the usc of the function $\eta$ and \eqref{eq4.Arut2}.(ii),
    $$
    \alpha\le\limsup_k\eta(f(x_k))\le\eta(\alpha)\,,$$
    so that, by \eqref{eq2.eta}, $\alpha =0$, that is,
    $$\lim_{k\to\infty}f(x_k)= 0\,.$$

    Now, by \eqref{eq4.Arut2}.(i),  for every $d\in P$,
    \begin{align*}
      \gamma d(x_{n+k},x_{n})&\le \gamma[d(x_{n+k},x_{n+k-1})+ \dots+d(x_{n+1},x_{n})]\\
      &\le f(x_{n+k-1})-f(x_{n+k})+\dots+f(x_{n})-f(x_{n+1})\\
      &=f(x_n)-f(x_{n+k})\,,
    \end{align*}
   so that
    \begin{equation}\label{eq5.Arut2}
   \forall d\in P,\;\;   \gamma d(x_{n+k},x_{n})\le f(x_n)-f(x_{n+k})\,,\end{equation}
    for all $n\in\mathbb{N}_0$ and $k\in\mathbb{N}.$

    Since the sequence $(f(x_k))_{k\in\mathbb{N}_0}$ is Cauchy, for every $\varepsilon>0$ there exists $n_\varepsilon\in\mathbb{N}_0$ such that
    $$
    0\le  f(x_n)-f(x_{n+k})<\gamma \varepsilon\,,$$
    for all $n\ge n_\varepsilon$ and all $k\in\mathbb{N}.$  Taking into account \eqref{eq5.Arut2}, it follows
    $$
    d(x_{n+k},x_n)<\varepsilon\,,$$
    for all $n\ge n_\varepsilon$,  $k\in\mathbb{N}$ and all $d\in P$. This shows that the sequence $(x_n)$ is right $K$-Cauchy, and so, by the right $K$-completeness
    of the quasi-uniform space $(X,\mathcal U)$, it is $\tau(\mathcal U)$-convergent to some $\bar x \in X$.

    Now $$x_n\xrightarrow{\tau(\mathcal U)} \bar x,   \;\; f(x_n)\to 0\,,$$ and \eqref{eq1a.Arut2} imply $f(\bar x)=0=\inf f(X)$.

    By Propositions \ref{p.qu-qms} and \ref{p.quc-qm}, the function $d(\cdot,x_0)$ is $\tau(\mathcal U)$-lsc for every $d\in P$, so that, by \eqref{eq5.Arut2} (with $n=0$ and $k=n$),
    $$\gamma d(x_n,x_0)\le f(x_0)-f(x_n)\,,$$
    whence
    $$
    \gamma d(\bar x,x_0)\le \gamma \liminf_{n}d(x_n,x_0)\le f(x_0)\,,$$
    which implies \eqref{eq3.Arut2}.(ii).
\end{proof}

As a consequence of Theorem \ref{t2.Arut}, one obtains the following result proved by Arutyunov and Gel'man \cite{arut-gel09}.
\begin{corol}
 Let $X,\,P$ and $\mathcal U=\mathcal U_P$ be as in Theorem \ref{t2.Arut}. Suppose that the proper function $f:X\to[0,\infty]$ satisfies
condition \eqref{eq1a.Arut2} and there exist $\lambda>0$ and $0<\mu<1$ such that for every $x\in X$ there exists $x'\in X$ such that
\begin{equation}\label{eq2.Arut3}\begin{aligned}
  {\rm (i)}\;\; &\forall d\in P,\;\;  d(x',x)\le \lambda f(x),\\
   {\rm (ii)}\;\;&f(x')\le \mu f(x).
\end{aligned}\end{equation}

Then for every $x_0\in \dom f$ there exist $\bar x\in X$ such that
\begin{equation}\label{eq3.Arut3}\begin{aligned}
  {\rm (i)}\;\; &f(\bar x)=0=\inf f(X);\\
  {\rm (ii)}\;\; &\forall d\in P,\;\; d(\bar x,x_0)\le\frac {\lambda}{1-\mu}\,f(x_0)\,.
\end{aligned}\end{equation}
\end{corol}\begin{proof}  Multiplying the inequality (i) in \eqref{eq2.Arut3} by $1-\mu$ and (ii)  by $\lambda$ and adding them, one obtains
$$
\lambda  f(x')+(1-\mu)d(x',x)\le \lambda f(x)\,,$$
or, equivalently,
$$
  f(x')+\frac{1-\mu}{\lambda}\,d(x',x)\le  f(x)\,,$$
i.e. \eqref{eq2.Arut2}.(i)  holds with $\gamma=(1-\mu)/\lambda$. The inequality \eqref{eq2.Arut2}.(ii) also holds with $\eta(t)=\mu t,\, t\in\mathbb{R}_+$.

Consequently, Theorem \ref{t2.Arut} yields \eqref{eq3.Arut3}.
\end{proof}

\begin{remark} It is obvious that any sequentially lsc or with sequentially closed graph function $f:X\to[0,\infty]$ satisfies \eqref{eq1a.Arut2}.
In the case of an lsc function the result follows from the Ekeland Variational Principle (see Theorem \ref{t.Arut-Fqg}).
\end{remark}
\begin{example}[\cite{arut-gel09}] Let $(X_1,d_1) $ and $(X_2,d_2)$ be metric spaces, $\emptyset \ne A\subseteq X_1,\, h:X_1\to X_2$ continuous
and $g:A\to Y$ with the graph closed in $X_1\times X_2$. Then the function $f:X_1\to[0,\infty]$ given by
$$
f(x)=\begin{cases}
  d_2(h(x),g(x)) \quad&\mbox{for }\; x\in A\\
  \infty   \quad&\mbox{for }\; x\in X_1\smallsetminus A
\end{cases}$$
satisfies \eqref{eq1a.Arut2}.
\end{example}

Indeed, if
$$
x_n\to x\;\mbox{ and }\; f(x_n)\to 0\,,$$
then, by the continuity of $h$,  $h(x_n)\to h(x)$. Since $d_2(h(x_n),g(x_n))\to 0$, it follows $g(x_n)\to h(x)$,
so that $g(x)=h(x)$ (because $g$ has closed graph).

Hence, $$f(x)=d_2(h(x),g(x))=d_2(h(x),h(x))=0\,.$$

\textbf{Acknowledgements.} The author expresses his warmest thanks to reviewers for their substantial and pertinent remarks and suggestions that led to an essential improvement of the paper.

\end{document}